\documentclass[10pt, reqno]{amsart}

\usepackage{amssymb,latexsym}
\usepackage{eucal}
\usepackage{tikz}
\usepackage{comment}

\numberwithin{equation}{section}

\newtheorem{thm}{Theorem}[section]

\newtheorem{lemma}{Lemma}[section]
\newtheorem{theorem}{Theorem}[section]

\newcommand{\al}{\alpha}

\newcommand{\ph}{\varphi}

\newcommand{\wtl}{\widetilde}
\newcommand{\wht}{\widehat}
\DeclareMathOperator{\core}{core}

\DeclareMathOperator{\St}{\textsf{St}}
\newcommand{\brr}{\bar {\mathrm{r}}}
\newcommand{\rr}{\mathrm{r}}
\newcommand{\x}{\times}

\newcommand{\cncm}{connected component}
\newcommand{\ED}{elementary deformation}

\begin{document}

\title[On joins and intersections of subgroups in free groups]
 {On joins and intersections of subgroups in free groups}
\author{S. V. Ivanov}
 \address{Department of Mathematics\\
 University of Illinois \\
 Urbana\\   IL 61801\\ U.S.A.} \email{ivanov@illinois.edu}
\thanks{Supported in part by the NSF under grant DMS 09-01782.}

\begin{abstract}
 We study  graphs of (generalized) joins and intersections of finitely generated subgroups of a free group.
 We show how to disprove a lemma of Imrich and M\"uller on these graphs and how to repair this lemma.
\end{abstract}

\maketitle

\section{Introduction}

Suppose  that $F$ is a free group of finite rank, $\rr(F)$ denotes the rank of $F$, and  $\brr(F) := \max
(\rr(F)-1,0)$ is the reduced rank of $F$.  Let  $H, K$ be finitely
generated subgroups of  $F$ and let    $\langle  H, K \rangle$ denote the the subgroup generated by $H, K$, called the
{\em join} of $H, K$.  Hanna Neumann \cite{N1}  proved
that $ \brr (H \cap K) \le 2 \bar \rr(H) \bar \rr(K)$  and
conjectured that  $\brr (H \cap K) \le  \bar \rr(H) \bar \rr(K)$. This problem, known as the Hanna Neumann conjecture
on subgroups of free groups, was solved in the affirmative by Friedman \cite{F} and Mineyev \cite{M},
see also Dicks's proof \cite{D1}. Relevant results  and generalizations of this conjecture can be found in
\cite{D}, \cite{DIv}, \cite{DIv2}, \cite{I99}, \cite{I01}, \cite{I08}, \cite{WN}, \cite{St}.

Imrich and M\"uller \cite{IM} introduced the reduced rank $\brr (\langle  H, K \rangle)$ of the join $\langle  H, K \rangle$ in this context and attempted to prove the following inequality
\begin{equation}\label{inq11}
\bar \rr(H  \cap K) \le   2 \bar \rr(H) \bar \rr(K) - \brr (\langle  H, K \rangle) \min(\bar \rr(H), \bar \rr(K))
\end{equation}
under the assumption that if $H^*, K^*$ are free factors of $H, K$, resp., then the equality
$H^*\cap K^* = H  \cap K$ implies that $H= H^*$ and $K = K^*$. We remark that this inequality provides a stronger
than  Hanna Neumann conjecture's bound for $\bar \rr(H  \cap K)$ whenever
$\brr (\langle  H, K \rangle) > \max(\bar \rr(H), \bar \rr(K))$ and this looks quite remarkable. We also note that
\eqref{inq11} was an improvement of an earlier result of Burns \cite{B}, see also \cite{N2}, \cite{S}, stating
that
\begin{equation*}
\bar \rr(H  \cap K) \le   2 \bar \rr(H) \bar \rr(K) - \min(\bar \rr(H), \bar \rr(K)) .
\end{equation*}

Later Kent \cite{K} discovered a serious gap in the proof of a key lemma of Imrich--M\"uller \cite[p.~195]{IM} and gave his own proof to the inequality \eqref{inq11} under the weakened assumption that $H  \cap K \ne \{ 1 \}$. Kent
\cite[p.~312]{K} remarks that the lemma in \cite{IM} ``would be quite useful, and though its proof is incorrect, we do
not know if the lemma actually fails". In this note we  give an example that shows that the lemma of \cite{IM} is indeed false.
On positive side, we suggest a repair for this lemma so that the arguments of Imrich--M\"uller \cite{IM} could be saved. Our approach seems to be of independent interest and can be outlined as follows. Given subgroups $H, K$ of a free group $F$ with $H  \cap K \ne \{ 1 \}$, we modify $H, K, F$ by certain deformations of Stallings graphs of $H, K, F$ so that the modified subgroups
 $\wtl H, \wtl K$ would satisfy $\brr(\wtl H) = \brr( H)$, $\brr(\wtl K) = \brr( K)$, $\bar \rr(\wtl H  \cap \wtl K) = \bar \rr(H  \cap K)$, and
 $\brr (\langle \wtl  H, \wtl  K \rangle) \ge \brr (\langle \wtl  H, \wtl  K \rangle)$. Furthermore, our modification 
is done  so that Stallings graphs of subgroups $\langle \wtl  H, \wtl  K \rangle$, $\wtl H, \wtl K$  have vertices of degree
2 or 3 only and every vertex of the graph of   $\langle \wtl  H, \wtl  K \rangle$ of degree 3 has a preimage of degree 3
in the graph of $\wtl  H$ or $ \wtl  K$. These properties mean that the lemma of Imrich--M\"uller \cite[p.~195]{IM} holds  in a modified setting and the arguments of Imrich--M\"uller are retained otherwise. As an application of this strategy developed in  Sect.~3, we will prove Theorem~\ref{th1} in Sect.~4.
Theorem~\ref{AB} in Sect.~5 deals with a natural question related to our main construction
used in  proofs of key Lemmas~\ref{ED}--\ref{mn}.

\section{A Counterexample to Imrich--M\"uller Lemma}

Similarly to Stallings \cite{St}, see also \cite{D}, \cite{K}, \cite{KM}, we consider
finite graphs associated with finitely generated subgroups of a free group $F$.
We consider the ambient free group $F$  as the fundamental group of a finite connected graph $U$, $F = \pi_1(U)$, without vertices of degree 1. Let $H, K$ be  finitely generated subgroups of $F$ and let $X, Y$ denote
finite graphs associated with $H, K$, resp. Recall that there are locally injective graph maps $\ph_X : X \to U$,
$\ph_Y : Y \to U$. Conjugating $H, K$ if necessary, we may assume that the graphs $X, Y$ have no vertices of degree 1, i.e.,   $\core (X) = X$ and $\core (Y) = Y$, where $\core(\Gamma )$ is the subgraph of a graph $\Gamma $ consisting of all edges
that can be included into circuits of $\Gamma $.

Let $S(H,K)$ denote a set of representatives of those double cosets $H tK \subseteq F$,  $t \in F$,
that have the property $H t Kt^{-1} \ne \{ 1\}$. Recall that the connected components of the core $W := \core (X \times_U Y)$ of the pullback $X \times_U Y$  of graph maps $\ph_X : X \to U$, $\ph_Y : Y \to U$
are in bijective correspondence with elements of the set $S(H,K)$,  see \cite{D}, \cite{KM},   \cite{WN}. Hence, we can write
\begin{equation*}
W = \bigvee_{t \in S(H,K)} W_t ,
\end{equation*}
where $\bigvee$ denotes a disjoint union. In addition, if $W_t$ is a \cncm\ of $W$ that corresponds to $t \in  S(H,K)$, then
\begin{equation*}
\brr ( H \cap tKt^{-1} ) =   | E W_t | -   | V W_t | ,
\end{equation*}
 where $E \Gamma $ is the set of (nonoriented) edges of  a graph $\Gamma $, $V \Gamma $ is the set of vertices of $\Gamma $ and $|A|$ is the cardinality of a set $A$. Using notation $\brr (\Gamma) := | E \Gamma | -   | V \Gamma |$, we have
 \begin{equation*}
\brr (W) = \sum_{t \in S(H,K) }  \brr(   W_t ) =  \sum_{t \in S(H,K) } \brr ( H \cap tKt^{-1} ) .
\end{equation*}
 For $S_1 \subseteq S(H,K)$, $S_1 \ne  \varnothing$, denoting  $W(S_1) := \bigvee_{s \in S_1} W_s$, we obtain
 \begin{equation}\label{inq2}
\brr (W(S_1) ) = \sum_{s \in S_1 }  \brr(   W_s ) =  \sum_{t \in S_1 } \brr ( H \cap sKs^{-1} ) = \brr(H, K, S_1) .
\end{equation}

Let $\al_X : W(S_1)  \to X$,  $ \al_Y :  W(S_1)  \to Y$ denote the restrictions on  $W(S_1) \subseteq   X \x_U Y$ of the pullback projection maps
$$
\bar \al_X : X \x_U Y \to X ,   \quad  \bar \al_Y :  X \x_U Y \to Y ,
$$
resp. Consider the pushout  $P$ of these maps $\al_X : W(S_1)  \to X$,  $ \al_Y :  W(S_1)  \to Y$. The corresponding
pushout maps are denoted $\beta_X : X \to P$,  $\beta_Y : Y \to P$. We also consider the Stallings graph $Z$ corresponding to the subgroup $\langle H, K, S_1 \rangle$ of $F$. It is clear from the definitions that there are
graph maps $\gamma : P \to Z$ and $\delta : Z \to U$ such that the diagram depicted in Fig.~1 is commutative.

\begin{center}
\begin{tikzpicture}[scale=.57]
\draw [-latex](.5,.5) -- (1.5, 1.5);
\draw [-latex](.5,-.5) -- (1.5, -1.5);
\draw [-latex](2.5,1.5) -- (3.5, 0.5);
\draw [-latex](2.5,-1.5) -- (3.5, -0.5);
\draw [-latex](4.5,0) -- (5.5, 0);
\draw [-latex](6.5,0) -- (7.5, 0);
\node at (-.2,0) {$W(S_1)$};
\node at (2,2) {$X$};
\node at (2,-2) {$Y$};
\node at (.6,1.4) {$\alpha_X$};
\node at (.6,-1.4) {$\alpha_Y$};
\node at (3.5,1.4) {$\beta_X$};
\node at (3.5,-1.4) {$\beta_Y$};
\node at (5.,0.5) {$\gamma$};
\node at (7.,0.5) {$\delta$};
\node at (4,0) {$P$};
\node at (6,0) {$Z$};
\node at (8,0) {$U$};
\node at (4.5,-3) {Figure 1};
\end{tikzpicture}
\end{center}

As was found out by Kent \cite[Fig.~1]{K}, one has to distinguish between $P$ and $Z$ as $P \ne Z$ in general. Since the map $\delta : Z \to U$ is locally injective and $P \ne Z$, we see that
the map $\gamma : P \to Z$  need not be locally injective, hence, $\gamma$ factors out through a sequence of edge
foldings and $\brr(P) \ge \brr(Z)$. According to Kent \cite{K}, the erroneous  identification $P = Z$ is the source of a mistake in a key lemma of  Imrich--M\"uller \cite[p.~195]{IM}. Recall that this lemma claims, in our terminology, that if
$Z$  is an {\em almost trivalent} graph, i.e., every vertex of $Z$ has degree 3 or 2, then every vertex of $Z$ of degree 3, has a preimage of  degree 3 in $X$ or $Y$.  Kent \cite{K} explains in detail a mistake in the proof of this lemma and comments that, while
the proof of lemma of \cite[p.~195]{IM} can be somewhat corrected if one replaces $Z$ by $P$, the subsequent arguments of
Imrich--M\"uller rely on the property that  $Z=U$ is  trivalent, the property that the graph $P$ does not possess.  Furthermore, Kent \cite[p.~312]{K}  points out that the lemma in \cite{IM} ``would be quite useful, and though its proof is incorrect, we do not know if the lemma actually fails".

We now present an example that shows that the lemma of \cite{IM} does fail. Consider subgroups
\begin{align}
H & = \langle b_1b_2 a_2 b_2^{-1} b_1^{-1}, \ a_1^2, \  a_1b_1b_3 a_3  (a_1b_1b_3)^{-1} \rangle , \\
K & = \langle b_1b_2 a_2 b_2^{-1} b_1^{-1}, \ a_1^3,  \  a_1 b_1b_3 a_3 (a_1b_1b_3)^{-1} \rangle
\end{align}
of a free group $F= \pi_1(U)$ whose graphs $ X, Y, Z$, resp., are depicted in Fig.~2, where the base vertices of
$ X, Y, Z$ are dashed circled.

\begin{center}
\begin{tikzpicture}[scale=.57]
\draw  (-2,0)[fill = black]circle (0.06);
\draw  (-1,0)[fill = black]circle (0.06);
\draw  (0,0)[fill = black]circle (0.06);
\draw  (2,0)[fill = black]circle (0.06);
\draw  (3,0)[fill = black]circle (0.06);
\draw  (4,0)[fill = black]circle (0.06);
\draw  (0,0)[dashed]circle (0.2);

\draw[-latex](-3.1,1) -- (-2.95,1);
\draw[-latex](0.95,1) -- (1.1,1);
\draw[-latex](4.95,1) -- (5.1,1);

\draw[-latex](1.1,-1) -- (0.95,-1);
\draw[-latex](-1.5,0) -- (-1.6,0);
\draw[-latex](-0.5,0) -- (-0.6,0);
\draw[-latex](2.5,0) -- (2.6,0);
\draw[-latex](3.5,0) -- (3.6,0);

\draw(-2,0) -- (0,0);
\draw  (-3,0) ellipse (1 and 1);
\draw  (1,0) ellipse (1 and 1);
\draw(2,0) -- (4,0);
\draw  (5,0) ellipse (1 and 1);

\node at (-3,1.5) {$a_2$};
\node at (1,1.5) {$a_1$};
\node at (1,-1.5) {$a_1$};
\node at (5,1.5) {$a_3$};
\node at (-1.5,0.5) {$b_2$};
\node at (-0.5,0.5) {$b_1$};
\node at (2.5,0.5) {$b_1$};
\node at (3.5,0.5) {$b_3$};
\node at (-1.5,-1.5) {$X$};

\begin{scope}[shift={(12,0)}]
\draw  (-2,0)[fill = black]circle (0.06);
\draw  (-1,0)[fill = black]circle (0.06);
\draw  (0,0)[fill = black]circle (0.06);
\draw  (2,0)[fill = black]circle (0.06);
\draw  (3,0)[fill = black]circle (0.06);
\draw  (4,0)[fill = black]circle (0.06);
\draw  (0,0)[dashed]circle (0.2);

\draw[-latex](-3.1,1) -- (-2.95,1);
\draw[-latex](0.95,1) -- (1.1,1);
\draw[-latex](4.95,1) -- (5.1,1);
\draw[-latex](-1.5,0) -- (-1.6,0);
\draw[-latex](-0.5,0) -- (-0.6,0);
\draw[-latex](2.5,0) -- (2.6,0);
\draw[-latex](3.5,0) -- (3.6,0);

\draw  (1,-1)[fill = black]circle (0.06);
\draw[-latex](1.73,-.7) -- (1.6,-.8);
\draw[-latex](.3,-.7) -- (.2,-.6);
\node at (2.1,-1.1) {$a_1$};
\node at (-.1,-1.1) {$a_1$};
\node at (-1.5,-1.5) {$Y$};

\draw(-2,0) -- (0,0);
\draw  (-3,0) ellipse (1 and 1);
\draw  (1,0) ellipse (1 and 1);
\draw(2,0) -- (4,0);
\draw  (5,0) ellipse (1 and 1);

\node at (-3,1.5) {$a_2$};
\node at (1,1.5) {$a_1$};

\node at (5,1.5) {$a_3$};
\node at (-1.5,0.5) {$b_2$};
\node at (-0.5,0.5) {$b_1$};
\node at (2.5,0.5) {$b_1$};
\node at (3.5,0.5) {$b_3$};
\end{scope}

\draw  (5,-4) ellipse (1 and 1);
\draw  (9,-4) ellipse (1 and 1);
\draw  (7,-8) ellipse (1 and 1);

\draw  (1,0) ellipse (1 and 1);
\draw(7.,-7) -- (7,-5.5);
\draw (5.7,-4.7) -- (7,-5.5);
\draw (8.3,-4.7) -- (7,-5.5);

\draw[-latex](7,-6.3) -- (7,-6.1);

\draw[-latex](6.38,-5.12)--(6.18,-5.0);
\draw[-latex](7.65,-5.12)--(7.8,-5.04);

\draw  (5,0) ellipse (1 and 1);
\draw  (7,-5.5)[fill = black]circle (0.06);
\draw  (5.7,-4.7)[fill = black]circle (0.06);
\draw  (8.3,-4.7)[fill = black]circle (0.06);
\draw  (7,-7)[fill = black]circle (0.06);
\draw  (7,-7)[fill = black]circle (0.06);
\draw  (7,-7)[dashed]circle (0.2);
\draw(-2,0) -- (0,0);

\draw[-latex](4.95,-3) -- (5.1,-3);

\draw[-latex](8.95,-3) -- (9.1,-3);

\draw[-latex](6,-8.1) -- (6,-7.9);
\node at (5,-2.5) {$a_2$};
\node at (9,-2.5) {$a_3$};
\node at (5.5,-8) {$a_1$};
\node at (6,-5.5) {$b_2$};
\node at (8,-5.5) {$b_3$};
\node at (7.5,-6.2) {$b_1$};
\node at (4.5,-6.5) {$Z$};

\node at (7,-10.25) {Figure 2};
\end{tikzpicture}
\end{center}

\noindent
It is easy to check that $S(H,K)$ has a single element, say $S(H,K) = \{ 1\}$, and
$$
H \cap K = \langle b_1b_2 a_2 b_2^{-1} b_1^{-1}, a_1^6,  a_1b_1b_3 a_3 (a_1b_1b_3)^{-1}  \rangle .
$$
Furthermore,  the graphs $W, P$ look like those in Fig.~3.  Hence, we can see that $Z$ is a trivalent graph that  has a vertex of degree 3, which is the center of $Z$, without a preimage of degree 3 in $X \vee Y$, as desired.

\begin{center}
\begin{tikzpicture}[scale=.52]
\draw  (-2,0)[fill = black]circle (0.06);
\draw  (-1,0)[fill = black]circle (0.06);
\draw  (0,0)[fill = black]circle (0.06);
\draw  (2,0)[fill = black]circle (0.06);
\draw  (3,0)[fill = black]circle (0.06);
\draw  (4,0)[fill = black]circle (0.06);
\draw  (0,0)[dashed]circle (0.2);

\draw[-latex](-3.1,1) -- (-2.95,1);
\draw[-latex](0.95,1) -- (1.1,1);
\draw[-latex](4.95,1) -- (5.1,1);

\draw[-latex](1.1,-1) -- (0.95,-1);
\draw[-latex](-1.5,0) -- (-1.6,0);
\draw[-latex](-0.5,0) -- (-0.6,0);
\draw[-latex](2.5,0) -- (2.6,0);
\draw[-latex](3.5,0) -- (3.6,0);

\draw(-2,0) -- (0,0);
\draw  (-3,0) ellipse (1 and 1);
\draw  (1,0) ellipse (1 and 1);
\draw(2,0) -- (4,0);
\draw  (5,0) ellipse (1 and 1);

\node at (-3,1.5) {$a_2$};
\node at (1,1.5) {$a_1$};
\node at (1,-1.5) {$a_1^5$};
\node at (5,1.5) {$a_3$};
\node at (-1.5,0.5) {$b_2$};
\node at (-0.5,0.5) {$b_1$};
\node at (2.5,0.5) {$b_1$};
\node at (3.5,0.5) {$b_3$};
\node at (-1.5,-1.5) {$W$};

\draw  (9.5,1) ellipse (1 and 1);
\draw  (13.5,1) ellipse (1 and 1);
\draw  (11.5,-3) ellipse (1 and 1);

\draw  (1,0) ellipse (1 and 1);
\draw(11.5,-2) --  (10.2,0.3);
\draw(11.5,-2) -- (12.8,0.3);

\draw[-latex](11.2,-1.5) -- (11.1,-1.3);
\draw[-latex](11.8,-1.5) -- (11.9,-1.3);

\draw[-latex](10.57,-0.4)--(10.48,-.2);
\draw[-latex](12.4,-0.4)--(12.52,-0.2);

\draw  (5,0) ellipse (1 and 1);
\draw  (10.85,-0.88)[fill = black]circle (0.06);
\draw  (12.15,-0.88)[fill = black]circle (0.06);

\draw  (10.2,0.3)[fill = black]circle (0.06);
\draw  (12.8,0.3)[fill = black]circle (0.06);
\draw  (11.5,-2)[fill = black]circle (0.06);
\draw  (11.5,-2)[fill = black]circle (0.06);
\draw  (11.5,-2)[dashed]circle (0.2);
\draw(-2,0) -- (0,0);

\draw[-latex](9.45,2) -- (9.6,2);

\draw[-latex](13.45,2) -- (13.6,2);

\draw[-latex](10.5,-3.1) -- (10.5,-2.9);
\node at (9.5,2.5) {$a_2$};
\node at (13.5,2.5) {$a_3$};
\node at (10,-3) {$a_1$};
\node at (10.,-0.5) {$b_2$};
\node at (13,-0.5) {$b_3$};
\node at (12.5,-1.6) {$b_1$};
\node at (10.5,-1.6) {$b_1$};
\node at (9,-1.5) {$P$};

\node at (6.5,-4) {Figure 3};
\end{tikzpicture}
\end{center}

\section{Fixing the Lemma of Imrich and M\"uller}

We now  discuss how to do certain deformations over the graphs $W(S), X$, $Y, P$, $Z, U$ to achieve the situation when
$P = Z =U$,  $U$ is almost trivalent and lemma of \cite[p.~195]{IM} would hold for $Z$.

Our idea could be illustrated by the remark that, when studying graphs  $W(S)$, $X$, $Y, P$, $Z, U$, or corresponding subgroups, we can replace the ambient group $F = \pi_1(U)$ by $\wht F = \pi_1(Z)$, i.e., we can  replace the graph $U$ by $Z$.
We can go further and replace the group $F = \pi_1(U)$ with $\wtl F = \pi_1(P)$ by using the pushout graph $P$ in place of $U$.

Either of these replacements $U \to Z$, $U \to P$  results in obvious cosmetic changes  to subgroups $H, K$, to sets $S(H,K)$, $S_1 \subseteq S(H,K)$, and to subgroups $ H \cap sKs^{-1}$, $s \in S_1$, $\langle  H, K, S_1 \rangle$.

Either of these replacements $U \to Z$, \ $U \to P$ preserves the graphs $W(S)$, $X, Y$, $P$ and the maps
$\al_X, \al_Y, \beta_X, \beta_Y$.  The replacement  $U \to Z$ turns $\delta$ into the identity map and the replacement
$U \to P$ turns both $\delta, \gamma$ into the identity maps. Otherwise, the structure of the diagram depicted on Fig.~1 is retained. In particular, the ranks $\brr(H), \brr(K)$, $\brr(W(S))$ do not change but the rank $\brr(P)$ could increase
as $\brr(P) \ge \brr(Z)$ originally.

For future references, we record this replacement in the following.

\begin{lemma}\label{rpl}    Replacing the graph  $U$ in the diagram depicted in Fig.~1 by $P$ and
changing the maps $\gamma : P \to Z$, $\delta: Z \to U$ by $\gamma = \delta = \mbox{id}_{U}$, resp., and
changing the maps $\ph_X : X \to U$, $\ph_Y : Y \to U$ by $\beta_X : X \to P$, $\beta_Y : Y \to P$, resp.,
preserve the properties that $W(S_1) = \bigvee_{s \in S_1}  W_{s}$  consists of connected components of
$\core(X \x_{U} Y)$ and that $P = X \vee_{W(S_1)} Y$.  In particular, this replacement does not
change the ranks $\brr(X), \brr(Y), \brr(W(S_1))$ but could increase  $\brr(Z)$.
\end{lemma}

\begin{proof}  If $(e_1, e_2) $ is an edge of $W(S_1) \subseteq  \core(X \x_{U} Y)$, where $\ph_X \al_X((e_1, e_2)) = e_1$ is an edge of $X$ and  $\ph_Y \al_Y((e_1, e_2)) = e_2$ is an edge of $Y$, then the edges $e_1, e_2$ are identified in $P$, whence, $(e_1, e_2) $ is also an edge of  $X \x_{P} Y$. Since  $\core (W(S_1) ) = W(S_1)$, it follows that
$\core(X \x_{P} Y)$ contains all edges $(e_1, e_2) $ of $W(S_1)$, hence,  $\core(X \x_{P} Y)$ will contain the graphs naturally isomorphic to \cncm s of the graph $W(S_1) =  \bigvee_{s\in S_1} W_s$. Therefore, the original pushout
$P = X \vee_{W(S_1)} Y$ will also be preserved.

Since $\brr(P) \ge \brr(Z)$ for the original graphs $P, Z$ and since $P = Z = U$ after the replacement $U \to P$, we see that the rank $\brr(Z)$ could increase after the replacement.
\end{proof}

Suppose that $v$ is a vertex of $U$. Subdividing the edges incident to $v$ if necessary, we may assume that
every oriented edge $e$ of $U$ whose terminal vertex, denoted $e_+$, is $e_+ = v$ has the initial vertex, denoted $e_-$, of degree 2, $\deg e_- = 2$. Hence, we may assume that $U$ contains a subgraph $\St(v)$ isomorphic to a star
with $\deg v$ rays whose center is $v$. Let $T$ be a tree whose vertices of degree 1 are in bijective correspondence with
vertices of degree 1 of $\St(v)$ whose set we denote by $V_1 \St(v)$. Taking $\St(v)$ out of $U$ and putting the tree $T$ in place of $\St(v)$, using the bijective correspondence to identify vertices of degree 1 of $(U \setminus \St(v)) \cup V_1\St(v)$ and those of $T$, results in a transformation of $U$ which we call an {\em elementary deformation} of $U$ around $v$ by means of $T$.
It is clear that the obtained graph $U_T : = (U \setminus \St(v))\cup T$ is homotopically equivalent to $U$ and
$\pi_1(U_T)$ is isomorphic to $F = \pi_1(U)$.  We lift this \ED\ of the star around $v$ in $U$ to all the graphs $X, Y, W, Z$ by  replacement of stars around  preimages of the vertex $v$ by trees isomorphic to  suitable subtrees of $T$ and change, accordingly, the maps $\al_X, \al_Y$,  $\beta_X, \beta_Y$,
$\ph_X, \ph_Y$. The new graphs and maps obtained this way we denote by $X_T, Y_T$, $W(S_1)_T$, $P_{T},  Z_T$,
$\al_{X_T},  \al_{Y_T}$,  $\beta_{X_T}, \beta_{Y_T}$,
$\ph_{X_T}, \ph_{Y_T}$, $\gamma_T, \delta_T$.

Clearly, $\brr(Q_T) = \brr(Q)$, where $Q \in \{ X, Y, W, Z, U\}$.
As far as  the new pushout graph $P_T$ is concerned,  we can only claim that  $\brr(P_{T}) \ge \brr(P)$
which would be analogous to the following.

\begin{lemma}\label{ZT}  Let $P = Z = U$, let $\gamma$, $\delta$ be identity maps and let the graphs
 $X_T, \ldots$,  $U_T$ be obtained from  $X, \ldots$,  $U$  by an \ED\ around a vertex $v \in VU$ by means of a tree $T$.
 Then the map $\delta_T : Z_T \to U_T$  is an isomorphism, i.e., the graphs
$Z_T, U_T$ are naturally isomorphic. Furthermore, the restriction of the map
 $\delta_T\gamma_T : P_{T} \to U_T$ is bijective on the set $ \gamma_T^{-1}\delta_T^{-1} (U_T \setminus T)$ and the subgraph $ \widehat T :=   \gamma_T^{-1}\delta_T^{-1} (T)$ of $P_{T}$ is connected. In particular, $\brr(P_{T}) \ge \brr(P)$  and the equality holds if and only if $ \widehat T$ is a tree.
\end{lemma}

\begin{proof}  It follows from the definitions that if we collapse edges of $T \subset U_T$ into a point then we obtain the  graph $U$ back. Similarly, collapsing lifts of the edges of $T$ in $ Q_T$, $Q \in \{ X, Y, W(S_1), P \}$, into points, we obtain the original graph $Q$. This observation, together with the local injectivity of the map $\delta_T$, implies that
$Z_T = U_T$ and that the graph $ \widehat T :=   \gamma_T^{-1}\delta_T^{-1} (T)$ is connected.

Furthermore, the graph  $P_{T}$ consists of a subgraph isomorphic to $ (U_T \setminus T ) \cup V_1T   $, where $V_1T $ is the set of vertices of $T$ of degree 1,  along with the graph $\widehat T=   \gamma_T^{-1}\delta_T^{-1} (T)$ which is mapped by $ \delta_T\gamma_T$ to $T$. Surjectivity of the restriction of $\delta_T\gamma_T$ on $\widehat T$ follows from connectedness of $\widehat T$.
\end{proof}

The graph $\widehat T$ of Lemma~\ref{ZT} can be regarded as a ``blow-up" of the vertex $v \in VU$.  If   $\widehat    T$ turns out to be a tree, then  our attempt to nontrivially ``blow up" the vertex $v$ is unsuccessful. On the other hand, if  $\widehat   T $ is not a tree,  then we can  increase $\brr (U)$ by invoking  Lemma~\ref{rpl} and picking $P_{T}$ in place of  $U$.

It is of interest to note that even when $P = Z = U$ and $\brr(P_{T}) = \brr(P)$, i.e., $ \widehat T$ is a tree,
the tree $ \widehat T$ might look different from  $T$, in this connection, see Lemma~\ref{ED} and its proof.

The following is due to Kent \cite{K}. We provide  a proof for completeness.

\begin{lemma}\label{knt} Every vertex  of $P$ of degree at least $3$ has a preimage in $X \vee Y$  of degree at least $3$. 
\end{lemma}

\begin{proof}  Arguing on the contrary, assume that $u \in V P$ has degree $\deg u >2$ and every  lift  $v \in VX \vee VY$ of $u$ has degree 2. Let $v, v' \in VX \vee VY$ be two arbitrary lifts of $u$. It follows from the definitions of  $ P, W(S_1)$ that there is a sequence of vertices  $v_1 = v , \dots,  v_k = v' $ in $VX \vee VY$ with the following properties:

(a) The vertices $v_1 = v , \dots,  v_k = v' $ are mapped by
$\beta_X,  \beta_Y$ to $u$ and have degree 2;

(b)  For every $i =1, \dots, k-1$, the vertices $v_i$, $v_{i+1}$ belong to distinct connected components of $X \vee Y$;

(c)  For every $i =1, \dots, k-1$,  there is a vertex $w_i \in VW$ of degree 2 such that either $\alpha_X(w_i) = v_i$,  $\alpha_Y(w_i) = v_{i+1}$ if $v_i \in VX$ or $\alpha_Y(w_i) = v_i$,  $\alpha_X(w_i) = v_{i+1}$ if $v_i \in VY$.

It is clear from these properties and from the definitions of the graphs  $ P, W(S_1)$ that the vertices $v_1 = v , \dots,  v_k = v' \in VX \vee VY $ will be identified in the pushout $P$ so that the resulting vertex will have degree 2. Since $v, v' \in VX \vee VY$ were chosen arbitrarily, it follows that the degree of $u \in V P $ is also 2. This contradiction to  $\deg u >2$ proves our claim.
\end{proof}

We say that $T$ is a {\em trivalent tree} if every vertex of $T$ has degree 1 or 3.

\begin{lemma}\label{ED}   Suppose that $P = Z = U$ and $v_0$ is a vertex of $U$ of degree at least 4.
Then there is an \ED\ of $U$ around $v_0$     by means of a trivalent tree $T$ such that either  $\brr(P_{T}) > \brr(P)$
or  the following are true: $\brr(P_{T}) = \brr(P)$, the subgraph $ \widehat T =   \gamma_T^{-1}\delta_T^{-1} (T)$ of $P_{T}$ is a tree
and
$$
\sum_{v \in V P_{T} } \max(\deg v -3, 0) < \sum_{v \in V P } \max(\deg v -3, 0) .
$$
\end{lemma}

\begin{proof}  Denote $\deg v_0 = m_0 >3$. Let $e_1, \dots, e_{m_0}$ be all of the oriented edges of $U$ that end in $v_0$.
By Lemma~\ref{knt}, there is a vertex $v_1 \in VX \vee VY$, say $v_1 \in VX$, such that $\beta_X(v_1) = v_0$ and $\deg v_1 = m_1$, where $3 \le m_1 \le m_0$.   Reindexing if necessary, we may assume that if $f_1, \dots , f_{m_1} $ are all of the edges of  $X$ that end in $v_1$, then $\beta_X(f_i) = e_i$, $i =1, \dots, m_1$.

Let $T$ be a trivalent tree with $m_0$ vertices of degree 1.
Let $U_T$ denote  the graph obtained from $U$ by an \ED\ around $v_0$ by means of $T$.

Let $u_i$ denote the vertex of $T$ which becomes $(e_i)_-$ in $U_T$,   $i =1, \dots, m_0$.
Here we are assuming that $\deg(e_i)_- = 2$ for every $i$ as in the definition of an \ED\ around $v$.
It follows from Lemma~\ref{ZT} that if  the subgraph $\widehat T = \gamma_T^{-1} \delta_T^{-1}(T)  $  is not a tree for some $T$, that is, if  $\brr(\widehat T) \ge 0$,  then $\brr(\widehat Z_T) > \brr(\widehat Z)$ and our lemma is proven. Hence, we may assume that   $\widehat T $ is a tree for every trivalent tree $T$.

Suppose that $\widehat T $  contains no vertex of degree $m_0$ for some $T$. Then, obviously,
$$
\sum_{v \in V U_T} \max( \deg v -3, 0)  < \sum_{v \in VU} \max( \deg v -3, 0)
$$
and our lemma is true.

Thus we may suppose that, for every trivalent  tree $T$, $\widehat T$ is a tree which contains a vertex $\widehat u$  of degree $m_0$, i.e., $\widehat T$ is homeomorphic to a star.

 It follows from the definitions of the graphs $P_{T},  \widehat T$ that the minimal subtree $S$ of $T$ that contains the vertices $u_1, \dots, u_{m_1}$  (recall $u_i = (e_i)_-$ in $P_{T}$) is isomorphic to
 a subtree  $\widehat S$  of  $\widehat T$.
Indeed, a copy of $S$ will show up in $X_T$ in place of a star around the vertex $v_1$ of $X$, hence,
a copy $\widehat  S$ of $S$ will also show up in $P_T$ as a subgraph of $\widehat T$.
 Since $T$ is trivalent and
 $\widehat T$ is a tree that has $m_0$ vertices of degree 1, one vertex $\widehat u$  of degree $m_0$ and other vertices of degree 2, it follows that  $\widehat S$ contains a single vertex of degree $\ge 3$ which implies $m_1 = 3$.
Since $T$ is an arbitrary trivalent tree, we may pick a tree $T = T_{12}$ that has adjacent vertices $w_{12}$, $w_3$ of degree 3 such that $w_{12}$ is adjacent to both $u_1, u_2$ and $w_3$ is  adjacent to $u_3$,
see Fig. 4(a). Note that, in this case,  $S$  is  the subtree   of $T = T_{12}$  that contains vertices $u_1, u_2$, $u_3$, $w_{12}, w_3$
and has 4 edges that connect these vertices, see Fig. 4(a).
\begin{center}
\begin{tikzpicture}[scale=.64]

\draw  (0,0)[fill = black]circle (0.06);
\draw  (-1,1)[fill = black]circle (0.06);
\draw  (-1,-1)[fill = black]circle (0.06);
\draw  (2,0)[fill = black]circle (0.06);
\draw  (3,1)[fill = black]circle (0.06);
\draw(0,0) -- (-1,1);
\draw(0,0) -- (-1,-1);
\draw(0,0) -- (2,0);
\draw(2,0) -- (3,1);
\draw[dashed](2,0) -- (3,-.5);
\draw  (2.5,-1.5)[fill = black]circle (0.06);
\draw  (4.4,-0.5)[fill = black]circle (0.06);
\node at (3.5,-1) {$\ldots$};

\node at (-1,1.5) {$u_1$};
\node at (-1,-1.5) {$u_2$};
\node at (0.25,.5) {$w_{12}$};
\node at (1.75,.5) {$w_{3}$};
\node at (3,1.5) {$u_3$};
\node at (2,-1.5) {$u_4$};
\node at (5,-0.5) {$u_{m_0}$};
\node at (1,1.5) {$T_{12}$};
\node at (1.25,-2.5) {Figure 4(a)};

\begin{scope}[shift={(9,0)}]

\draw  (0,0)[fill = black]circle (0.06);
\draw  (-1,1)[fill = black]circle (0.06);
\draw  (-1,-1)[fill = black]circle (0.06);
\draw  (2,0)[fill = black]circle (0.06);
\draw  (3,1)[fill = black]circle (0.06);
\draw(0,0) -- (-1,1);
\draw(0,0) -- (-1,-1);
\draw(0,0) -- (2,0);
\draw(2,0) -- (3,1);
\draw[dashed](2,0) -- (3,-.5);
\draw  (2.5,-1.5)[fill = black]circle (0.06);
\draw  (4.4,-0.5)[fill = black]circle (0.06);

\node at (3.5,-1) {$\ldots$};

\node at (-1,1.5) {$u_1$};
\node at (-1,-1.5) {$u_3$};
\node at (0.25,.5) {$w_{13}$};
\node at (1.75,.5) {$w_{2}$};
\node at (3,1.5) {$u_2$};
\node at (2,-1.5) {$u_4$};
\node at (5,-0.5) {$u_{m_0}$};
\node at (1,1.5) {$T_{13}$};
\node at (1.25,-2.5) {Figure 4(b)};
\end{scope}

\end{tikzpicture}
\end{center}

\noindent Since the image $\widehat  w_{12}$ of  $w_{12}$ in
$\widehat S$  has degree 3, we have $\deg \widehat  w_{12} = m_0$ in $\widehat T_{12}$.
This, in particular, means that every preimage of $w_3$ in $\widehat T_{12}$ has degree 2 and adjacent to a  preimage of  $w_{12}$.

Consider a vertex $u \in V Q$, $Q \in \{ X, Y \}$,  such that $\beta_Q(u) = v_0$.
Let $g_1, \dots, g_k$ be all of the oriented edges of $Q$ that end in $u$. We claim that the set $\{ \beta_Q(g_1), \dots, \beta_Q(g_k) \}$ may not contain both
$e_3$ and $e_j$, where $j >3$. Arguing on the contrary, assume that
\begin{equation}\label{inc}
\{ e_3, e_j \} \subseteq   \{ \beta_Q(g_1), \dots, \beta_Q(g_k)  \}  .
\end{equation}
Then a star neighborhood of $u$ in $Q$ would turn in $Q_{T_{12}}$ into a tree isomorphic
to a subtree $S_u$ of $T_{12}$ which contains vertices $w_3, u_3, u_j$, see Fig.~4(a).
Clearly, either $\deg w_3 = 3$ in ${S_u}$  or  $\deg w_3 = 2$ in ${S_u}$  and then $S_u$ contains no vertex $w_{12}$. In either case, a copy of $S_u$ is present in $\widehat T_{12}$ which is impossible because, as we saw above,  every preimage of $w_3$ in $\widehat T_{12}$ has degree 2 and adjacent to a  preimage of  $w_{12}$. Thus the inclusion \eqref{inc} is impossible.

Recall that we can pick any trivalent tree $T$, in particular, we can pick a tree $T = T_{13}$ that has adjacent vertices $w_{13}$, $w_2$ of degree 3 such that $w_{13}$ is adjacent to both $u_1, u_3$ and $w_2$ is  adjacent to $u_2$, see   Fig.~4(b).

Repeating the above argument with indices 2 and 3 switched, we can show that
$$
\{ e_2, e_j \} \not\subseteq \{ \beta_Q(g_1), \dots, \beta_Q(g_k)  \}
$$
for every $j > 3$.

Similarly, switching indices 1 and 3,  we prove that
$$
\{ e_1, e_j \} \not\subseteq \{ \beta_Q(g_1), \dots, \beta_Q(g_k)  \}
$$
for every $j > 3$.

Now we see that for every vertex $u \in VX \vee VY$ such that  $\beta_Q(u) = v_0$, where
$Q \in \{ X, Y \}$, the edges $ g_1, \dots, g_k $ of $Q$ that end in $u$ have the following property: either
$$ \{ \beta_Q(g_1), \dots, \beta_Q(g_k)  \} \subseteq
\{ e_1, e_2, e_3 \} \quad \mbox{or} \quad
 \{ \beta_Q(g_1), \dots, \beta_Q(g_k)  \} \subseteq
\{ e_4, \dots, e_{m_0} \}  .
$$
Since the edges of $X \vee Y$ are identified in the pushout  $P = X \vee_{W(S_1)}  Y$ if and only if their $\beta_X$-, $\beta_Y$-images in $P$ are equal, it follows that the terminal vertex of $e_1, e_2, e_3$ must be different in $P$ from the terminal vertex of an edge $e_j$ where $j \ge 4$. This contradiction to the definition of the edges  $e_1, \dots, e_{m_0}$ of $U$ completes the proof.
\end{proof}

We are now ready to prove our key lemma.

\begin{lemma}\label{mn}     Suppose that the graphs $X, Y, W(S_1), P, Z, U$ and corresponding maps
$\al_X,  \ldots, \delta$ are defined  as in Fig.~1. Then there exists a finite sequence $\tau$ of alternating replacements
as in Lemma~\ref{rpl} and \ED s as in Lemma~\ref{ED} that result in graphs
$X^\tau, Y^\tau$, $W(S_1)^\tau$, $P^\tau$,  $Z^\tau, U^\tau$ that have the following properties: $\brr(X^\tau) = \brr(X)$,
$\brr(Y^\tau) = \brr(Y)$, $\brr( W(S_1)^\tau) = \brr(W(S_1))$, \ $P^\tau =  Z^\tau = U^\tau$, $\brr(P^\tau) \ge \brr(P)$, every vertex of $U^\tau$ has degree  3 or 2 and every vertex of degree 3 of $U^\tau$ has a preimage of degree 3 in $X^\tau$ or  $Y^\tau$.
\end{lemma}

\begin{proof} Applying Lemma~\ref{rpl}, we may assume that $P = Z = U$. If every vertex of $U$ has degree 2 or 3, then our claim holds true as follows from Lemma~\ref{knt}. Hence, we may assume that $U$ contains a vertex $v_0$ of
degree at least 4. In view of Lemma~\ref{ED}, by checking all possible \ED s of $U$ around $v_0$, we can find an \ED\ by means of a trivalent tree $T$
such that either $\brr(P_{T}) > \brr(P)$ or $\brr(P_{T}) = \brr(P)$ and
$$
\sum_{v \in VP_{T} } \max(\deg v -3, 0) < \sum_{v \in V P } \max(\deg v -3, 0) .
$$
Invoking Lemma~\ref{rpl}, we replace $U_T$ with $P_{T}$ and, completing one cycle of changes in graphs $X, Y, W(S_1), P, Z, U$, we rename $Q := Q_T$, where $Q \in \{ X, Y, W(S_1), P, Z, U \}$, and start over.

Observe that the rank  $\brr(Z) = \brr( \langle H, K, S_1 \rangle) $ is bounded above by the number of generators
 $\rr(H) + \rr(K) + |S_1|$
and this bound does not change as we perform cycles of changes of graphs $X, Y, W(S_1), P, Z, U$,
because $|S_1|$ is equal to  the number of connected components of $W(S_1)$ and $\brr(H) = \brr(X)$,  $\brr(K) = \brr(Y)$. This implies that the number
$$
\sum_{v \in VZ } \max(\deg v -3, 0)
$$
is bounded above by $ 2\brr(Z)  \le 2(\rr(H) + \rr(K) + |S_1|)$. Thus the total number of cycles will not exceed
$2(\rr(H) + \rr(K) + |S_1|)^2$ and our lemma is proved.
\end{proof}

Note that the arguments of the proof of Lemma~\ref{mn} provide an algorithm that deterministically constructs
the desired graphs $X^\tau, Y^\tau$, $W(S_1)^\tau$, $P^\tau$,  $Z^\tau, U^\tau$ in polynomial space
(of size of input which are graphs $X, Y, W(S_1)$, $P, Z, U$ along with associated maps $\al_X, \ldots, \delta$).

\begin{lemma}\label{sum}
$\brr(Z) \le  \brr(X) +  \brr(Y)$.
\end{lemma}

\begin{proof} In view of Lemma~\ref{mn}, we may assume that Stallings graphs $X, Y, Z$ of subgroups
$H, K, \langle H, K, S_1 \rangle$ satisfy the conclusion of Lemma~\ref{mn}. Then  all  of the vertices  of $X,  Y, Z$   have degree 2 or 3.
 Hence, $2\brr(U) =  | V_3 U|$, where $ V_3 U$ is the set of vertices of degree 3 of $U$.
 Similarly, we have $2\brr(X) =  | V_3 X |$ and $2\brr(Y) =  | V_3 Y |$. Now application of  Lemma~\ref{knt} yields
 $|V_3 Z | \le | V_3 X | + | V_3 Y | $ which implies
  the desired inequality $\brr(Z) \le \brr(X) +  \brr(Y)$.
\end{proof}

\section{Applications}

As an application of Lemma~\ref{mn},  we state and prove a couple of specific inequalities for
reduced ranks of the generalized intersections  and joins of subgroups in free groups.

\begin{thm}\label{th1} Let $H, K$ be finitely generated subgroups of a free group $F$.
Let $S(H,K) \subseteq F$ denote a set of representatives of those double cosets $H tK \subseteq F$,  $t \in F$,
that have the property $H t Kt^{-1} \ne \{ 1\}$,  and let $S_1 \subseteq S(H,K)$ be nonempty.
Then $\brr (\langle  H,   K, S_1 \rangle)  \le  \brr( H) + \brr( K)$ and
\begin{align}\notag
 \brr(H, K, S_1) & :=   \sum_{s \in S_1}  \brr(H \cap s K s^{-1}) \\ \label{inq1t}
 & \le
  2  \brr(H)  \brr(K) - \brr (\langle  H, K, S_1 \rangle) \min(\brr(H),  \brr(K)) .
\end{align}
Moreover,
\begin{align}\label{inq2t}
\! \brr(H, K,  S_1) \! \le     \tfrac 12 ( \brr(H) \! + \! \brr(K)\! - \! \brr(\langle  H, K,  S_1 \rangle) )
( \brr(H) \! + \! \brr(K)\! -\! \brr(\langle  H,  K,  S_1 \rangle)\! + \! 1) .
\end{align}
\end{thm}

Note that the inequality \eqref{inq1t} is a strengthened version of \eqref{inq11} and
this strengthening is analogous to a strengthened version of
the Hanna Neumann conjecture introduced by  Walter Neumann \cite{WN}.
We also remark that the inequality \eqref{inq1t} in the cases when $S_0 = \{ 1 \}$ and  $S_0 = S(H,K)$ is due to Kent \cite{K}
and the inequality  \eqref{inq2t} is shown by Dicks \cite{D2} who also obtained other inequalities.

It is worthwhile  to mention that the natural question on the existence of a bound that would be
a stronger version of \eqref{inq1t}, in which the first term would have the coefficient 1 in place of 2
and the second negative term would contain the factor $\brr (\langle  H, K, S_1 \rangle)$ with some coefficient, has a negative solution. Indeed, according to \cite{I12}, there are finitely generated subgroups $H, K$ of a free group $F$
such that   $\brr(H, K, S(H,K)) = \brr(H)  \brr(K) >0$ and $\brr (\langle  H, K, S(H,K)) >C$ for any constant $C >0$.

\begin{proof}
In view of Lemma~\ref{mn}, we may assume that Stallings graphs $X, Y, Z$ of subgroups
$H, K, \langle H, K, S_1 \rangle$, resp., and graphs $W(S_1)^\tau$, $P^\tau$ satisfy the conclusion of Lemma~\ref{mn}.
Recall that $\brr(H) = \brr(X)$, $\brr(K) = \brr(Y)$ and $\brr(\langle H, K, S_1 \rangle ) \le  \brr(Z) = \brr(U)$,

As before, if $Q$ is a  graph whose vertices have degree 2 or 3, then $V_3 Q$ denotes the set of vertices of $Q$ of degree 3. Recall that  $2 \brr(Q) = |V_3 Q |$.

For every $v \in VU$, denote
$$
k_v := |V_3X \cap \beta_X^{-1}(v)|, \quad \ell_v := |V_3Y \cap \beta_Y^{-1}(v)|, \quad m := |V_3 W(S_1) |, \quad
n := |V_3 U | .
$$

We also let $V_3U := V_B U \vee V_X U\vee V_Y U$, where $V_B U$ is the set of vertices of $U$ that have preimages
of degree 3 in both $X$ and $Y$, $V_X U$ is the set of vertices of $U$ that have preimages
of degree 3 in $X$ only, and $V_Y U$ is the set of vertices of $U$ that have preimages
of degree 3 in $Y$ only.  Denote $n_B := |V_B U|$,  $n_X := |V_X U|$, $n_Y := |V_Y U|$.
We let  $V_3 X = V_{31} X \vee V_{32} X$, where $V_{31} X := \beta_X^{-1}(V_X U)$ and  $V_{32} X := \beta_X^{-1}(V_B U)$. Similarly, we let $V_3 Y = V_{31} Y \vee V_{32} Y$, where $V_{31} Y := \beta_Y^{-1}(V_Y U)$ and  $V_{32} Y := \beta_Y^{-1}(V_B U)$. Denote
$$
k := |V_3X|, \   k_1 := |V_{31} X|, \ k_2 := |V_{32} X|,  \  \ell := |V_3Y|, \   \ell_1 := |V_{31} Y|, \ \ell_2 := |V_{32} Y| .
$$
Clearly,
$
k = k_1 +k_2 , \   \ell = \ell_1 +\ell_2 , \  n = n_B+n_X +n_Y, \ k_1 \ge n_X, \ \ell_1 \ge n_Y$, and
$$
m  \le \sum_{v \in VU} k_v \ell_v = \sum_{v \in V_BU} k_v \ell_v .
$$

We now continue with arguments similar to those of Imrich-M\"uller \cite{IM} that follow the proof of their lemma.
Since $\sum_{v \in V_BU} \ell_v = \ell_2$ and $\ell = \ell_1 +\ell_2 $, we obtain
\begin{align}\notag
m & \le \sum_{v \in V_BU} k_v \ell_v =  k\ell - \sum_{v \in V_BU} (k -k_v) \ell_v  - k \ell_1 \\ \notag
 &  \le k \ell - \sum_{v \in V_BU} (k -k_v)  - k \ell_1 =  k \ell -  kn_B + k_2  - k \ell_1    \\  \label{in3}
 &  \le k \ell - k(n_B + n_Y -1 ) .
\end{align}
Switching $X$ and $Y$, we analogously obtain
\begin{align}\label{in4}
m & \le k \ell - \ell (n_B + n_X -1 ) .
\end{align}

Assume that $n_B = 0$. Then $\brr(W(S_1)) = 0$ and the inequality \eqref{inq1t} is equivalent to
\begin{align*}
0 \le 2 \brr(X) \brr(Y) -  \min( \brr(X), \brr(Y))\brr(Z)
\end{align*}
which is equivalent to $\brr(Z) \le 2 \max( \brr(X), \brr(Y))$. Since
$\brr(Z) \le \brr(X) + \brr(Y)$ by Lemma~\ref{sum},
it follows that   \eqref{inq1t}  holds true. A reference to Lemma~\ref{sum} also proves that
$$
\brr(\langle H, K, S_1 \rangle ) \le  \brr(Z) \le \brr(X)+ \brr(Y) =  \brr(H) +\brr(K) .
$$

Suppose that $n_B \ge 1$.  Since $n_B +n_X+n_Y = 2\brr(U) $, it follows that
$$
n_X+n_Y \le 2\brr(U)-1 .
$$
Hence,
$\min(n_X, n_Y) \le \brr(U)-1$ and $\max(n_B + n_X -1,  n_B + n_Y -1 ) \ge \brr(U)$.  Therefore, the inequalities
\eqref{in3}--\eqref{in4} imply that
\begin{align*}
2  \brr(W(S_1)) = m & \le 4 \brr(X) \brr(Y) - 2   \min( \brr(X), \brr(Y))\brr(Z) .
\end{align*}
Dividing by 2, we obtain  the required bound \eqref{inq1t}.
\medskip

Observe that  Lemma~\ref{sum}  together with the foregoing classification of vertices of degree 3 in graphs
$W(S), X, Y, P=Z=U$ make it possible to produce other inequalities for  $\brr(H, K, S_1) =  \brr(W(S_1)) $ that would involve $\brr(\langle H, K, S_1 \rangle ) = \brr(Z)$.  For example, when maximizing the sum
$\sum_{v \in VU} k_v\ell_v$ which gives an upper bound for $m = 2 \brr(W(S_1))$, we may assume that there is a single vertex
 $v$ such that both $k_v, \ell_v >0$, whence $n_B = 1$, $k_2 = k_v$ and $\ell_2 = \ell_v$.
 We may further  assume that $k_1 = n_X$, $\ell_1 = n_Y$. Indeed, if, say $k_1 > n_X$, then we could increase
 $k_v$ by  $k_1 - n_X$ and decrease $k_1$ by $k_1 - n_X$ making thereby $k_2\ell_2 = k_v\ell_v$ greater. Hence, to get an upper bound for $m$,  we can maximize the product $k_2\ell_2 = (k - k_1) (\ell -\ell_1)$ subject to $k_1 + \ell_1+ 1 = n$.
 Since $k, \ell, n$ are fixed positive even integers, equal to $2\brr(X)$, $2\brr(Y)$, $2\brr(Z)$, resp.,
 it follows that  the product $(k- k_1)(\ell -n+1 - k_1)$ has the maximum  at $k_1 = \tfrac 12( k-\ell +n)$. Since $k_1$ is an integer, it follows that  $(k- k_1)(\ell -n+1 - k_1)$ has a maximal value for integer $k_1$ when
$k_1 = \tfrac 12( k-\ell +n)$ or  $k_1 = \tfrac 12( k-\ell +n)-1$. This means that, unconditionally, we have
$$
m \le  (k- k_1)(\ell -n+1 - k_1)  \le \tfrac 12( k+\ell -n+2 ) \cdot \tfrac 12 (k+\ell -n )  .
$$
Equivalently,
 \begin{align*}
2  \brr(W(S_1)) = m & \le   ( \brr(X) + \brr(Y) - \brr(Z) +1) ( \brr(X) + \brr(Y) - \brr(Z))
\end{align*}
which implies the bound \eqref{inq2t}.  Theorem~\ref{th1} is proved.
\end{proof}

\section{One More Question}

One might wonder what would be the conditions that guarantee the existence of a nontrivial
``blow-up"  of a vertex $v_0$ of the  graph $P = Z =U$ by an \ED\ around $v_0$, as defined in Sect.~3.
Here we present a result that provides a criterion for the existence of such a ``blow-up" of a vertex $v_0$ of $U$, i.e.,
the existence of an \ED\ around a vertex $v_0$ by means of a tree $T$ such that $\brr(P_{T}) > \brr(P)$.

Assume that $P = Z =U$ and the maps $\gamma, \delta$ in Fig.~1 are identity maps on $U$. Let $v_0$ be a fixed vertex of $U$ and let $D$ denote the set of all oriented edges of $U$ that end in $v_0$.

A vertex $v$ of one of the  graphs $W(S_1), X, Y$ is called a {\em $v_0$-vertex} if the image of $v$ under  $\beta_X \al_X$, $\beta_X$, $\beta_Y$, resp., is $v_0$.

For every $v_0$-vertex $v \in V Q$, where $Q \in \{ X, Y, W(S_1) \}$,  define the set $D(v) \subseteq  D$ so that  $e \in D(v)$  if and only if there is an edge $f \in \vec E Q$ so that the terminal vertex $f_+$ of $f$ is $v$ and the image of the edge $f$ in $U$ is $e$.

Consider a  partition $D = A \vee B$ of the set $D $ 
into two nonempty subsets $A, B$.

A $v_0$-vertex $v \in V Q$, where $Q \in \{ X, Y, W(S_1) \}$, is said to be of {\em type $\{A, B\}$} if both intersections $D(v) \cap A$ and $D(v) \cap B$  are nonempty. The set of all vertices of type $\{A, B\}$  of a graph ${Q}$, where
 $Q \in \{ X, Y, W(S_1) \}$, is denoted $V_{\{A,B\}} Q$.

Consider a bipartite graph $\Psi(\{A, B \} )$ whose set of vertices
$$
V \Psi(\{A, B \})  = V_X \Psi(\{A, B \}) \vee V_Y  \Psi(\{A, B \} )
$$
consists of two disjoint parts
$$
V_X \Psi(\{A, B \}) :=  V_{ \{A, B \} }X \quad  \mbox{and } \quad
 V_Y  \Psi(\{A, B \} ):=  V_{\{ A, B \} }Y .
$$

Two vertices $v_X \in V_X \Psi(\{A, B \}) $ and
 $v_Y \in  V_Y \Psi(\{A, B \})$ are connected in $\Psi(\{A, B \})$ by an edge  if and only if there is a vertex $v \in W(S_1) $ such that
 $\alpha_{X}(v)  = v_X$,  $\alpha_{Y}(v)  = v_Y$ and $v \in V_{\{A, B \} }W(S_1)$.

\begin{theorem}\label{AB} Suppose $P = Z =U$, $v_0$ is a vertex of $U$.     The equality $\brr (P_{T}) = \brr (U_T) = \brr (U)$ holds for every \ED\ around $v_0$ by means of a tree $T$ if and only if the following conditions hold.
For every partition $D = A \vee B$  with nonempty $A, B$ if the graph $\Psi( \{ A, B \})$ contains $k$
connected components, then there exists a partition $D = \bigvee_{i=1}^{k+1} C_i$ such that, for every $i$, $C_i$ is nonempty and either $C_i \subseteq A$ or $C_i \subseteq
B$. Next, for every vertex $v \in V Q$, where $Q \in \{ X, Y \}$, such that $\beta_Q(v) = v_0$, the intersection
$D(v) \cap A$ is either empty  or $D(v) \cap A \subseteq C_{i_{vA}}$ for
some $i_{vA}$, $i_{vA} = 1, \dots, k+1$, and  the intersection $D(v) \cap B$ is either empty  or $D(v) \cap B \subseteq C_{i_{vB}}$ for some $i_{vB}, i_{vB} = 1, \dots, k+1$.
\end{theorem}

\begin{proof} First we introduce the notation we will need below. As above, let $T$
be a tree whose set of vertices of degree 1 is $D'$. If  $H' \subseteq D'$
is a nonempty subset of vertices of degree 1 in $T$, we let $M_T(H')$ denote the
minimal subtree of $T$ that contains $H'$. Clearly, the set of vertices of degree
$\le 1$ of $M_T(H')$ is $H'$.  Note that $M_T(H') = H'$ if $H'$  consists of a single
vertex. If $E \subseteq D$, then $E' \subseteq D'$ denotes the image of $E$
under the map $d \to d'$  for every $d \in D$.

For every vertex $v \in V Q$, where $Q \in \{ X, Y, W(S_1) \}$, consider a tree
$M_T( D(v) )$ which is isomorphic to the subtree
$M_T( D(v)')$ of $T$ and let
\begin{equation}\label{zeta}
\zeta_v :   M_T( D(v) ) \to T
\end{equation}
denote the natural monomorphism whose image is  $M_T( D(v)' )$.

If $w$ is a vertex of $W(S_1)$, then the tree $M_T( D( \alpha_{X}(w) ))$ contains a subtree
isomorphic to $M_T( D( w) )$  which we denote $M_{T, X}( D( w) )$.

Similarly, the tree  $M_T( D( \alpha_{Y}(w) ))$ contains a subtree isomorphic
to $M_T( D( w) )$, denoted $M_{T, Y}( D(w) )$.

It follows from the definitions that the graph $\widehat T = \gamma_T^{-1}
\delta_T^{-1} (T)$ can be described as follows.
$\widehat T $ is the union of graphs $M_T( D(v) )$ over all  $v \in V Q$, $Q \in
\{ X, Y \}$, which are identified along their subgraphs of the form
$M_{T, X}( D( w) )= M_{T, Y}( D( w) )$ over all $w \in V W(S_1)$.

Now assume that the equality $\brr (P_{T})= \brr (U_T)$ holds for every  tree $T$ or,
equivalently,  see Lemma~\ref{ZT},  the graph $\widehat T$ is a tree for every $T$. Note that
$\widehat T$ being a tree need not imply that $\wht T$ is naturally isomorphic to $T$.

Let $D =  A \vee B$ be a partition of $D$ with nonempty $A, B$. Our goal is to find
a partition $D = \bigvee_{i=1}^{k+1} C_i$  with the properties described in
Theorem~\ref{AB}.

Consider a trivalent tree $T = T(\{ A, B\})$ such that $T$ contains a nonoriented edge $e$ so
that the graph $T \setminus \{ e \}$ consists of two connected components $T_{eA}$, $T_{eB}$
such that  $A' \subseteq V T_{eA}$ and $B' \subseteq V T_{eB}$.

Note that the tree $M_T( D(v) )$, where $v \in V Q$, $Q \in \{ X, Y, W(S_1) \}$,
contains an edge $f$ with $\zeta_v(f) = e$, where $\zeta_v$ is defined by \eqref{zeta}, if and only if $v$ has type $\{ A, B \}$.
Hence, the set $\gamma_T^{-1}  \delta_T^{-1} (e) $  consists of images  of such edges $f$ of
$M_T( D(v) )$ with $\zeta_v(f)=e$.

Furthermore, let  $v, v' \in V_{ \{A, B \} } X    \vee  V_{ \{A, B \} } Y $ be two vertices that belong to the same connected component of $\Psi(\{A, B \} )$. Then there exists a
sequence  $v = v_1, v_2, \dots, v_{\ell} = v'$ of vertices in the graph  $\Psi(\{A, B \} )$ so that $v_i, v_{i+1}$ are connected by an edge. It follows from the definitions that the edges
$f_1, f_2, \dots, f_\ell $ such that  $\zeta_{v_i}(f_i)=e$, $i =1, \dots, \ell$,   are identified in $\widehat T$. Conversely, it follows from the definitions in a similar fashion that if
$v, v' \in V_{ \{A, B \} } X    \vee  V_{ \{A, B \} } Y $
and the edges $f, f'$ such that $\zeta_{v}(f)=e=\zeta_{v'}(f')$ are identified in
$\widehat T$, then $v, v'$ belong to the same connected component of $\Psi(\{A, B \} )$.
These remarks prove that there is a bijection between the set
$\gamma_T^{-1}  \delta_T^{-1} (e)  =  \{  f_1,
\dots, f_k    \} \subseteq  \widehat T $ and the set of connected components of the graph
$\Psi(\{A, B \} )$. Since  $\widehat T$ is a tree with $|D|$ vertices of degree 1 whose set we denote $D''$, it follows that the graph $\widehat T - (\gamma_T^{-1}  \delta_T^{-1} (e ) )$   splits into $k+1$ connected components which induce a partition $D'' = \bigvee_{i=1}^{k+1} C''_i$,  where for every $i$, $C''_i$ is nonempty and either $C''_i \subseteq A''$ or $C_i \subseteq B''$. Then $D = \bigvee_{i=1}^{k+1} C_i$,
where $C''_i  $ is the image of $C_i$ under the map $d \to d''$ for every $d \in D$, provides a desired partition for $D$
because,  for every $v_0$-vertex $v \in V Q$, where $Q \in \{ X, Y \}$, the vertices of
$A'' \cap D(v)'' \subseteq D''$ and those of $B'' \cap D(v)''$ are connected in $\widehat T - (\gamma_T^{-1}  \delta_T^{-1} (e  ) )$ by images of edges of $M_T( D(v) ) - \zeta_v^{-1}( \{ e \} )   $ of $\widehat T$. Therefore, $D(v) \cap A \subseteq C_{i_{vA}}$ for some $i_{vA}$ whenever the intersection $D(v) \cap A$ is not empty and
$D(v) \cap B \subseteq C_{i_{vB}}$ for some $i_{vB}$ whenever the intersection   $D(v) \cap B$ is nonempty.
\smallskip

Now we will prove the converse. Assume that for every partition $D = A \vee B$
there exists a  partition $D = \bigvee_{i=1}^{k+1} C_i$ with the properties of Theorem~\ref{AB}. Arguing on the contrary, suppose that there is a tree $T$ for which the graph  $\widehat T$ is not a tree. Pick a shortest circuit  $p$ in
$\widehat T$ of positive length. Let $e$ be a nonoriented edge of $T$ such that $e \in \delta_T\gamma_T(p)$.  The graph $T - \{ e \}$ consists of two
connected components $T_{eA}, T_{eB}$ which define a partition $D' = A' \vee B'$, where  $ A'  \subseteq T_{eA}$,  $ B'  \subseteq T_{eB}$, and both $ A',   B' $ are  nonempty. Let $k$ denote the number of connected components of the graph
$\Psi(\{A, B \} )$.  It is clear from Lemma~\ref{ZT} and the definitions that $k \ge 1$.  According to our assumption, there exists a partition
$D = \bigvee_{i=1}^{k+1} C_i$, where, for every $i$, $C_i  $ is nonempty, $C_i \subseteq A$ or $C_i \subseteq B$, and, for every $v_0$-vertex $v \in V Q$, where $Q \in \{ X, Y \}$, we have that  $D(v) \cap A \subseteq C_{i_{vA}}$  for some $i_{vA}$ whenever  $D(v) \cap A$  is nonempty and we have that   $D(v) \cap B \subseteq C_{i_{vB}}$  for some $i_{vB}$ whenever  $D(v) \cap B$  is nonempty.

As above, we observe that if  $v, v' \in V_{ \{A, B \} } X    \vee  V_{ \{A, B \} } Y $ are two vertices, then the edges $\zeta^{-1}_{v}(e)$, $\zeta^{-1}_{v'}(e)$ are identified in  $\widehat T$ if and only if $v, v'$ belong to the same connected component of the graph  $\Psi(\{A, B \} )$. This implies that the edges $\gamma_T^{-1}  \delta_T^{-1} (e ) \subseteq \widehat T$ are in bijective correspondence with connected components of $\Psi(\{A, B \} )$. Hence, there are $k$ edges in $\gamma_T^{-1}  \delta_T^{-1} (e )$. Denote
$\gamma_T^{-1}  \delta_T^{-1} (e ) =   \{ f_1,  \dots, f_k \}$.
Since identification of subgraphs
$$
M_{T, X}(D(w)) \subseteq M_{T}(D(\alpha_X(w) ))    \quad \mbox{and} \quad  M_{T, Y}(D(w)) \subseteq M_{T}(D(\alpha_Y(w) ))
$$
in process of construction of $\widehat T$
respects the  partition $D' = \bigvee_{i=1}^{k+1} C'_i$, it follows that the graph
$\widehat T$ consists of pairwise disjoint  subgraphs $S_1, \dots, S_{k+1}$, where
$$
\delta_T\gamma_T(S_j) \cap  D'= C'_j , \quad j = 1, \dots, k+1 ,
$$
which are joined by $k$ nonoriented  edges of  $ \gamma_T^{-1} \delta_T^{-1} (e )$. Collapsing subgraphs $S_1, \dots, S_{k+1}$ into points, we see that  $\widehat T$ turns into a tree with $k$ nonoriented edges and $k+1$ vertices.  Consequently, a shortest circuit in  $\widehat T$  of positive length may not contain any edge of
$ \gamma_T^{-1} \delta_T^{-1} ( e )$. This contradiction to the choice of the path $p$ and the edge $e$ in
 $  \delta_T\gamma_T(p)$ completes the proof of Theorem~\ref{AB}.
\end{proof}

\end{document}